\documentclass[11pt]{amsart}
\usepackage{tgpagella}
\usepackage{euler}
\usepackage[T1]{fontenc}
\usepackage{amssymb}
\usepackage{amsmath,amsthm}
\usepackage[hidelinks]{hyperref}
\usepackage{amsthm}
\usepackage[english]{babel}
\usepackage{mathrsfs}
\usepackage{eucal}

\newtheorem{main}{Theorem}

\newtheorem{theorem}{Theorem}[section]
\newtheorem{lem}[theorem]{Lemma}
\newtheorem{prop}[theorem]{Proposition}

\theoremstyle{definition}
\newtheorem{definition}[theorem]{Definition}

\def\e{\epsilon}
\def\la{\lambda}

\def\bb{\mathbb}

\def\om{\omega}

\def\bb{\mathbb}

\def\Cal{\mathcal}

\def\cb{\operatorname{cb}}

\DeclareMathOperator{\OS}{\mathcal{OS}}
\DeclareMathOperator*{\id}{{id}}
\DeclareMathOperator*{\CP}{{CP}}
\numberwithin{equation}{section}

\providecommand{\nor}[1]{\lVert #1 \rVert}

\begin{document}

\title{CP-stability and the local lifting property}

\author[T. Sinclair]{Thomas Sinclair}
\address{Department of Mathematics, Indiana University, Rawles Hall, 831 E 3rd St, Bloomington, IN 47405, USA}
\email{tsincla@purdue.edu}


\date{\today }
\dedicatory{}

\begin{abstract} The purpose of this note is to discuss the local lifting property in terms of an equivalent approximation-type property, CP-stability, which was formulated by the author and Isaac Goldbring for the purposes of studying the continuous model theory of C$^*$-algebras and operator systems.

\end{abstract}

\maketitle


\section{Statement of the Main Results}

The following definition first appears in \cite{gs1}.

\begin{definition}An operator system $X$ is said to be \emph{CP-stable} if for any finite dimensional subsystem $E\subset X$ and $\delta>0$ there is a finite-dimensional subsystem $E\subset S\subset X$ and $k, \e>0$ so that for every C$^*$-algebra $A$ and any unital linear map $\phi: S\to A$ with $\nor{\phi}_k<1+\e$ there exists a u.c.p.\ map $\psi: E\to A$ so that $\nor{\phi|_E - \psi}<\delta$.
\end{definition}

Let $u_1,\dotsc, u_n$ be the canonical generators of $C^*(\bb F_n)$ and let $W_n$ be the operator system spanned by the set $\{u_i^*u_j : 1\leq i,j\leq n+1\}$ where $u_{n+1} := 1$. The first result gives a quantitative version of CP-stability for the operator systems $W_n$ using work of Farenick and Paulsen \cite{fp12}.

\begin{main}\label{main1} The operator system $W_n$ is CP-stable. In particular, for any $\delta>0$ there exists $\e>0$ so that for any unital linear map $\phi: W_n\to A$ into an arbitrary unital C$^*$-algebra with $\nor{\phi}_{n+1}<1+ \e$, there is a u.c.p.\ map $\psi: W_n\to A$ so that $\nor{\psi - \phi}<\delta$.

\end{main}

\noindent Note that by \cite[Corollary 4.7]{CH} no such result can hold for the related generator subsystem of the \emph{reduced} C$^*$-algebra $C^*_\la(\bb F_n)$.

We say an operator system $X$ has the \emph{local lifting property} (LLP) of Kirchberg if for every unital C$^*$-algebra $A$, every ideal $J$ of $A$, and every u.c.p.\ map $\phi: X\to A/J$ and every finite-dimensional subsystem $E\subset X$ the restricted map $\phi|_E$ admits a u.c.p.\ lifting $\tilde\phi: E\to A$.\footnote{The definition we give here is termed the \emph{operator system local lifting property} (OSLLP) in \cite{kavruk, kptt} though for our purposes we will not make a distinction.}  It was shown in \cite{gs2} that for C$^*$-algebras CP-stability is equivalent to the local lifting property. 

Using operator system tensor product characterizations for exactness and the LLP (see \cite{kptt}), Kavruk showed that a finite-dimensional operator system has the LLP if and only if its dual system is exact \cite[Theorem 6.6]{kavruk}. We show that conversely, one can use the fact that the dual is exact (in the sense of admitting a nuclear embedding), i.e., that the operator system is CP-stable, to recover Kirchberg's tensor characterization of the LLP \cite{kptt, kir93}.

\begin{main}\label{main2} If $E$ is a finite-dimensional operator system which is CP-stable, then $E\otimes_{\min} \Cal B(\ell^2) \cong E\otimes_{\max} \Cal B(\ell^2)$ as operator systems.
\end{main}

\noindent Using techniques from \cite{pis96} or \cite{kavruk} Theorem \ref{main1} and Theorem \ref{main2} give a new proof of the fact (due to Kirchberg \cite{kir94}) that $C^*(\bb F_n)\otimes_{\min} \Cal B(\ell^2) \cong C^*(\bb F_n)\otimes_{\max} \Cal B(\ell^2)$.

The author is grateful to Isaac Goldbring for many stimulating discussions from which these ideas arose.

\section{Proofs of the Main Results}

The following result is due to Farenick and Paulsen \cite{fp12}: see the remarks after Definition 2.1 therein.

\begin{lem}\label{lem1} The ``covering'' map $\gamma_n: M_{n+1}\to W_n$ defined by $\phi(e_{ij}) = \frac{1}{n+1} u_i^*u_j$, where $u_1,\dotsc, u_{n+1}$ are defined as above, is u.c.p.\ and the kernel $J_{n+1}$ consists of all diagonal matrices in $M_{n+1}$ of trace zero.
\end{lem}

Remarkably, Farenick and Paulsen \cite[Theorem 2.4]{fp12} go on to show that:

\begin{theorem}[Farenick+Paulsen] \label{thm1} The map $\overline{\gamma_n}: M_{n+1}/J_{n+1}\to W_n$ is a complete order isomorphism where the quotient space $M_{n+1}/J_{n+1}$ is equipped with its canonical operator system structure as defined in \cite[Section 3]{kptt}.
\end{theorem}

The strategy of our proof of Theorem \ref{main1} will be to make use of the fact that matrix algebras are CP-stable.

\begin{lem}[Proposition 2.40 in \cite{gs1}] \label{lem2} Given $k$, for any $\delta>0$ there exists $\e>0$ so that for any C$^*$-algebra $A$ and any unital linear map $\phi: M_k\to A$ with $\nor{\phi}_k<1+\e$, there exists a u.c.p.\ map $\tilde\phi: M_k\to A$ so that $\nor{\tilde\phi - \phi}<\delta$.
\end{lem}

For the reader's convenience we provide a streamlined proof.

\begin{proof} Suppose by contradiction that there is some $\delta>0$ so that for every $n$ there is some unital linear map $\phi_n: M_k\to A_n$ into some C$^*$-algebra $A_n$ so that $\nor{\phi}_k< 1+ \frac{1}{n}$ so that $\nor{\psi - \phi}\geq \delta$ for any u.c.p.\ map $\psi: M_k\to A_n$. Fix an nonprinciple ultrafilter $\om$ on $\bb N$ and define $\Cal A := (A_n)_\om$ to be the ultrapower C$^*$-algebra associated to the sequence $(A_n)$ and $\om$ and $\widetilde{\Cal A} := \prod_n A_n$. Consider the map $\phi := (\phi_\bullet): M_k\to \Cal A$. Clearly $\phi$ is unital and $\nor{\phi}_k=1$ whence by \cite[Proposition 2.11]{paulsen} $\phi$ is $k$-positive. By Choi's theorem \cite[Theorem 3.14]{paulsen} $\phi$ is therefore u.c.p.\ and the proof of the Choi+Effros lifting theorem \cite[Theorem C.3]{BO} shows there is thus a u.c.p.\ lift $\tilde\phi: M_k\to \widetilde{\Cal A}$. However, this shows that the sequence $(\phi_n)$ is well-approximated by u.c.p.\ maps for $n\in\om$ generic, a contradiction.
\end{proof}

\begin{proof}[Proof of Theorem \ref{main1}] We begin by fixing $\delta>0$ and a unital C$^*$-algebra $A$. Suppose we have a unital linear map $\phi: W_n\to A$ with $\nor{\phi}_{n+1}< 1+ \e$ for some $\e>0$ sufficiently small and to be determined later. We will show that we can find a u.c.p.\ map $\psi: W_n\to A$ so that $\nor{\psi - \phi}<\delta$.

Let $\phi' := \phi\circ\gamma_n: M_{n+1}\to A$ which is again unital and linear with $\nor{\phi}_{n+1}<1+\e$. By Lemma \ref{lem2} we may choose $\e>0$ sufficiently small so that there is a u.c.p.\ map $\psi': M_{n+1}\to A$ so that $\nor{\psi' - \phi'}< \delta/16n^4$. Since $\phi'(e_{ii}) = \frac{1}{n+1}1_A$, we have that $\nor{\psi'(e_{ii}) - \frac{1}{n+1}1_A} < \delta/16n^4$ whence $b_i := \psi'(e_{ii})$ is uniformly invertible and positive. Let $B\in M_{n+1}(A)$ be the diagonal matrix such that $B_{ii} := b_i^{-1/2}$. Let $\Psi' := [\psi'(e_{ij})]\in M_{n+1}(A)$ be the Choi matrix associated to $\psi'$. Since $\Psi'$ is positive, so is $\Psi'' := B\Psi'B$, whence it defines a c.p.\ map $\psi'': M_{n+1}\to A$ via the reverse correspondence $\psi''(e_{ij}) := \Psi_{ij}''$. We can see manifestly that $\psi''(e_{ii}) = \frac{1}{n+1}1_A$ whence $\psi''$ is unital, $J_{n+1}\subset \ker(\psi'')$, and $\nor{\psi'' - \psi'} < \delta/4n^2$. 

Identifying $W_n$ with the quotient operator system $M_{n+1}/J_{n+1}$ by Theorem \ref{thm1}, since $J_{n+1}\subset \ker(\psi'')$ it follows by \cite[Proposition 3.6]{kptt} that there is u.c.p.\ map $\psi: W_n\to A$ so that $\sup_{i,j}\nor{\psi(u_i^*u_j) - \phi(u_i^*u_j)}< \delta/2n^2$. Alternatively, this is not difficult to see by setting $\psi := \psi''\circ\gamma_n^{-1}$ and unravelling the definition of the quotient operator system structure via the identification given by Theorem \ref{thm1}. In any case it follows by the small perturbation argument that $\nor{\psi - \phi}<\delta$, and we are done.
\end{proof}

Two formal weakenings of the LLP were introduced by the author and Isaac Goldbring: the \emph{local ultrapower lifting property} (LULP) 
\cite[Proposition 2.42]{gs1} and the \emph{approximate local lifting property} (ALLP) \cite[Definition 7.3]{gs2}.\footnote{The ALLP is implicitly formulated in the work of Effros and Haagerup \cite{eh}, where it is shown to be equivalent to the LLP.} Both definitions carry over straightforwardly to the category of operator systems. For instance, an operator system $X$ can be said to have the LULP if every u.c.p.\ map $\phi: X\to A_\om$ admits local u.c.p.\ lifts to $\ell^\infty(A)$. The following proposition is essentially contained in \cite{gs1,gs2}. We provide a sketch of the proof for the convenience of the reader.

\begin{prop} For an operator system $X$ the following statements are equivalent:

\begin{enumerate}
\item $X$ has the LLP;
\item $X$ has the ALLP;
\item $X$ has the LULP;
\item $X$ is CP-stable.
\end{enumerate}

\end{prop}

\noindent The equivalence of the first two statements essentially appears in the work of Effros and Haagerup \cite[Theorem 3.2]{eh}. We also remark that using the equivalence with the ALLP, it is easy to see that the LLP passes to inductive limits, noting that it suffices to check the ALLP only on a dense subalgebra.

\begin{proof} The equivalence of (3) and (4) is proved in  \cite[Proposition 2.42]{gs1}. The implication (1) $\Rightarrow$ (2)  is straightforward.  For (2) $\Rightarrow$ (3), we note that by the small perturbation we can require the approximate lifts to be unital, and we may also assume they are $\ast$-linear.  In conjunction with \cite[Corollary B.11]{BO} which shows that we can correct such an approximate lift to a u.c.p.\ map a controlled distance away (depending on the dimension of the domain), we can thus assume that the approximate lifts are u.c.p.\ from which the implication follows easily. We include a proof of (4) $\Rightarrow$ (2), though it closely follows the reasoning given in \cite[Proposition 7.7]{gs2}.

To this end, note that by the main result of \cite{RS} that for any finite-dimensional operator system, any u.c.p.\ map $\phi: E\to A/J$ admits $n$-positive unital liftings $\tilde\phi_n: E\to A$ for every $n$. Hence if $E$ was a finite dimensional subsystem of a CP-stable system $X$ and $\phi: X\to A/J$ was u.c.p.\ it would follow that for every $n$ there is a u.c.p.\ map $\psi: E\to A$ so that $\nor{\pi_J\circ\psi - \phi|_E}<\frac{1}{n}$, where $\pi_J: A\to A/J$ is the quotient $\ast$-epimorphism. Hence $X$ has the ALLP.

Finally, the implication (2) $\Rightarrow$ (1) follows from a foundational result of Arveson that liftable u.c.p.\ maps are closed in the point-norm topology: see \cite[Lemma C.2]{BO}.

\end{proof}

Let $\OS_n$ be the set of all complete isomorphism classes of $n$-dimensional operator systems. The set $\OS_n$ is naturally equipped with two complete metrics, the $\cb$-Banach distance and the weak metric: see \cite{gs2} for details. With the equivalence of LLP and CP-stability in hand, we give a new proof of a result of Kavruk \cite{kavruk}.

\begin{prop}[Kavruk]\label{prop-dual} A finite-dimensional operator system $E$ is exact if and only if the dual system $E^*$ is CP-stable.
\end{prop}

\begin{proof} It is well known (see \cite{pis95, gs2}) that $E$ is an exact $n$-dimensional operator system if and only if any sequence of unital $\ast$-linear maps $\phi_\alpha: E_\alpha\to E$ so that $\nor{\phi_\alpha}_k\to 1$ for all $k$, there is a sequence of unital maps $\psi_\alpha: E_\alpha\to E$ with $\nor{\psi_\alpha}_{\cb}\to 1$ with $\nor{\psi_\alpha - \phi_\alpha}\to 0$. Dualizing (noting by \cite[Proposition 2.1]{jungepisier} or \cite{bp91} that this is a well behaved operation) and applying a standard compactness argument, we see that this is implies that $E^*$ is CP-stable. The converse follows similarly by unravelling the definitions.
\end{proof}

In the category of operator systems, the correct treatment of tensor products has only recently appeared in the work of Kavruk, Paulsen, Todorov, and Tomforde \cite{kptt1, kptt}. We refer to these works for the basic definitions and properties of various operator system tensor products. Using these ideas we give a new proof of a famous and difficult theorem of Kirchberg \cite{kir94}. A short and particularly elegant proof of the same result in the context of operator spaces was given by Pisier \cite{pis96}. A second elementary proof was recently discovered by Farenick and Paulsen \cite{fp12}. (See also \cite{han, oz}.)

\begin{theorem}[Kirchberg]\label{thm-kir94} If $E$ is a finite-dimensional operator system which is CP-stable, then $E\otimes_{\min} \Cal B(\ell^2) = E\otimes_{\max} \Cal B(\ell^2)$.
\end{theorem}

\begin{lem}\label{lem-cover} Let $E$ be a finite-dimensional operator system with the LLP. Then for every $\e>0, k$ there is $n$ and a u.c.p.\ map $\phi: M_n^* \to E$ so that for any positive $x\in M_k(E\otimes_{\min} \Cal B(\ell^2))^+$ there is $\hat x\in M_k(M_n^*\otimes_{\min} \Cal B(\ell^2))^+$ positive so that $\nor{x - (\phi\otimes\id)_k(\hat x) }<\e$.
\end{lem}

\begin{proof} Let us fix $\e, k>0$. Using \cite[Lemma 8.5]{kptt} we may identify the positive cone $M_k(E\otimes_{\min} \Cal B(\ell^2))^+$ with the space $\CP(E^*, M_k(\Cal B(\ell^2)))$ of completely positive maps $\phi: E^*\to M_k(\Cal B(\ell^2))$. By Proposition \ref{prop-dual} $E^*$ is exact, so there are matricial operator systems $E_m\subset M_{\ell_m}$ and u.c.p.\ bijections $\phi_m: E^*\to E_m$ with $\nor{\phi_m^{-1}}_{\cb} \to 1$.

By pre-composition each map $\phi_m$ induces a map \[\Phi_{m,k}: \CP(M_{\ell_m}, M_k(\Cal B(\ell^2)))\to \CP(E^*, M_k(\Cal B(\ell^2)))\] which preserves unitality and is easily identified with the u.c.p.\ map \[(\phi_m^*\otimes\id)_k: M_k(M_{\ell_m}^*\otimes_{\min} \Cal B(\ell^2))\to M_k(E\otimes_{\min} \Cal B(\ell^2)).\](We are using that the minimal operator system tensor product is functorial: see \cite[Theorem 4.6]{kptt1}.)

Given a u.c.p.\ map $\psi: E^*\to \Cal B(H)$ we may pre-compose with $\phi_m^{-1}$ to obtain a unital, self-adjoint map  $\psi_m: E_m\to \Cal B(H)$ which we may isometrically extend to $\hat \psi_m: M_m\to \Cal B(H)$. As $\nor{\psi_m}_{\cb}\leq \nor{\phi_m^{-1}}_{\cb}\to 1$, by \cite[Corollary B.9]{BO} there is an approximating sequence to $(\hat \psi_m)$ consisting of u.c.p.\ maps $\upsilon_m: M_m \to \Cal B(H)$ with $\nor{\hat\psi_m - \upsilon_m}\leq 2(\nor{\phi_m^{-1}}_{\cb}-1)$. Via the identification with $\Phi_{m,k}$ it therefore follows that $(\phi_m^*\otimes\id)_k$ restricted to $M_k(M_{\ell_m}^*\otimes_{\min} \Cal B(\ell^2))^+$ is $\e$-surjective into $M_k(E\otimes_{\min} \Cal B(\ell^2))^+$ for $m$ sufficiently large.

\end{proof}

Note that $M_n$ is (trivially) exact and is CP-stable, whence $M_n^*$ shares these properties by Proposition \ref{prop-dual}. It would be interesting to know whether $M_n^*$ is a C$^*$-nuclear operator system in the sense of \cite{kavruk, kptt} --- as essentially noted in \cite[Theorem 6.7]{kavruk} this is predicted by Connes' embedding conjecture. (N.B.\ One's first instinct might be to conclude that C$^*$-nuclearity for all $M_n^*$ would force exactness and CP-stability to coincide for finite-dimensional operator systems; however, it is easy to find the basic fault in this argument.)

\begin{lem}\label{lem-hope} For all $n$, $M_n^*\otimes_{\min} \Cal B(\ell^2) = M_n^*\otimes_{\max} \Cal B(\ell^2)$. 
\end{lem}

\begin{proof} Since $\Cal B(\ell^2)$ has the WEP, by \cite[Lemma 6.1 and Theorem 6.7]{kptt} it suffices to check that $M_n^*\otimes_{\min} E = M_n^*\otimes_{\max} E$ for any finite-dimensional operator system $E$. Using \cite[Proposition 1.9]{fp12} we have that $(M_n^*\otimes_{\min} E)^* \cong M_n\otimes_{\max} E^* \cong M_n\otimes_{\min} E^*$ as operator systems. By the same \[M_n^*\otimes_{\min} E \cong (M_n^*\otimes_{\min} E)^{**}\cong (M_n\otimes_{\min} E^*)^*\cong M_n^*\otimes_{\max} E,\] and we are done.
\end{proof}

\begin{proof}[Proof of Theorem \ref{thm-kir94}] Given $\e, k>0$, by Lemma \ref{lem-cover} we can find $n$ such that there is a u.c.p.\ map $\phi: M_n^*\to E$ so that $(\phi\otimes\id)_k: M_k(M_n^*\otimes_{\min} \Cal B(\ell^2))^+\to M_k(E\otimes_{\min} \Cal B(\ell^2))^+$ is $\e$-surjective. By Lemma \ref{lem-hope} we have that $M_k(M_n^*\otimes_{\min} \Cal B(\ell^2))^+ = M_k(M_n^*\otimes_{\max} \Cal B(\ell^2))^+$. Since the maximal tensor norm is functorial \cite[Theorem 5.5]{kptt1}, it follows that $(\phi\otimes\id)_k$ maps $M_k(M_n^*\otimes_{\max} \Cal B(\ell^2))^+$ into $M_k(E\otimes_{\max} \Cal B(\ell^2))^+$. As $\e$ was arbitrary this shows that $M_k(E\otimes_{\min} \Cal B(\ell^2))^+\subset M_k(E\otimes_{\max} \Cal B(\ell^2))^+$, and we are done.

\end{proof}


\begin{thebibliography}{999999!}

\bibitem[BP91]{bp91} D.P. Blecher and V.I. Paulsen, {\it Tensor products of operator spaces}, J. Funct. Anal. {\bf 99} (1991) 262-292.

\bibitem[BO08]{BO} N.P. Brown and N. Ozawa, {\it C$^*$-Algebras and Finite-Dimensional Approximations}, Grad. Studies in Math. {\bf 88}, AMS, Providence, RI, 2008.

\bibitem[CH85]{CH} J. de Canni\`ere and U. Haagerup, {\it Multipliers of the Fourier algebras of some simple Lie groups and their discrete subgroups}, Amer. J. Math. {\bf 107} (1985) 455-500. 

\bibitem[EH85]{eh} E. Effros and U. Haagerup, \textit{Lifting problems and local reflexivity for C$^*$ algebras}, Duke Math. J. \textbf{52} (1985),103-128.

\bibitem[FP12]{fp12} D. Farenick, V. Paulsen, \emph{Operator system quotients of matrix algebras and their tensor products}, Math. Scand. {111} (2012) 210--243.

\bibitem[GS15a]{gs1} I. Goldbring and T. Sinclair, {\it On Kirchberg's embedding problem}, J. Funct. Anal. {\bf 269} (2015) 155-198.

\bibitem[GS15b]{gs2} I. Goldbring and T. Sinclair, {\it Omitting types in operator systems}, preprint, arXiv:1501.06395.

\bibitem[Ha14]{han} K.H. Han, {\it A Kirchberg type tensor theorem for operator systems}, preprint, arXiv:1409.1306.

\bibitem[JP95]{jungepisier} M. Junge and G. Pisier, \emph{Bilinear forms on exact operator spaces and $B(H)\otimes B(H)$}, Geom. Funct. Anal. (GAFA) \textbf{5} (1995), 329-363.

\bibitem[Ka14]{kavruk} A.S. Kavruk, \emph{Nuclearity related properties in operator systems}, Journal of Operator Theory \textbf{71} (2014), 95-156.

\bibitem[KP+11]{kptt1} A.S. Kavruk, V.I. Paulsen, I.G. Todorov, and M. Tomforde, {\it Tensor products of operator systems}, J. Funct. Anal. {\bf 261} (2011) 267-299.

\bibitem[KP+13]{kptt} A.S. Kavruk, V.I. Paulsen, I.G. Todorov, and M. Tomforde, {\it Quotients, exactness, and nuclearity in the operator system category}, Adv. Math. {\bf 235} (2013) 321-360.

\bibitem[Ki93]{kir93} E. Kirchberg, \emph{On non-semisplit extensions, tensor products, and exactness of group C$^*$-algebras}, Invent. Math. \textbf{112} (1993) 449-489.

\bibitem[Ki94]{kir94} E. Kirchberg, \emph{Commutants of unitaries in UHF algebras and functorial properties of exactness} J. reine angew. Math. {\bf 452} (1994) 39-77.

\bibitem[Oz13]{oz} N. Ozawa, {\it About the Connes embedding conjecture: algebraic approaches}, Jpn. J. Math. {\bf 8} (2013) 147-183.

\bibitem[Pa03]{paulsen} V.I. Paulsen, \textit{Completely bounded maps and operator algebras}, Cambridge Studies in Advanced Mathematics \textbf{78}, Cambridge University Press (2003).

\bibitem[Pi95]{pis95} G. Pisier, {\it Exact operator spaces}, in {\it Recent Advances in Operator Algebras (Orl\'eans, 1992)}, Ast\'erisque {\bf 232} (1995) 159-186. 

\bibitem[Pi96]{pis96} G. Pisier, {\it A simple proof of a theorem of Kirchberg and related results on C$^*$-norms}, J. Operator Theory {\bf 35} (1996) 317-335. 

\bibitem[RS89]{RS} A.G. Robertson and R.R. Smith, {\it Liftings and extensions of maps on C$^*$-algebras}, J. Operator Theory {\bf 21} (1989), 117-131.

\end{thebibliography}
\end{document}